\documentclass[12 pt]{article}%
\usepackage{amsmath, amsfonts, amsthm, color,latexsym}
\usepackage{amsmath}
\usepackage{amsfonts}
\usepackage{amssymb}
\usepackage{color, soul}
\usepackage{graphicx}%
\usepackage{enumerate}
\usepackage[utf8]{inputenc}
\usepackage[T1]{fontenc}
\allowdisplaybreaks[4]
\providecommand{\U}[1]{\protect\rule{.1in}{.1in}}
\newtheorem{theorem}{Theorem}[section]
\newtheorem{proposition}[theorem]{Proposition}
\newtheorem{corollary}[theorem]{Corollary}
\newtheorem{example}[theorem]{Example}

\newtheorem{remark}[theorem]{Remark}

\newtheorem{lemma}[theorem]{Lemma}
\newtheorem{final remark}[theorem]{Final Remark}
\newtheorem{definition}[theorem]{Definition}
\begin{document}

\title{\sc Absolutely $\gamma$ - summing polynomials and the notion of coherence and compatibility}
\author{Joilson Ribeiro\thanks{joilsonor@ufba.br}~ and Fabr\'icio Santos\thanks{Supported
by a CAPES Doctoral scholarship.\thinspace \hfill\newline\indent2010 Mathematics Subject
Classification: 46B45, 47L22, 46G25.\newline\indent Key words: Banach sequence spaces, ideals of homogeneous polynomials, linear stability, finitely determined.}}
\date{}
\maketitle

\begin{abstract} In this paper, we construct the abstract ideal of polynomials. We show this is an ideal of Banach and, in a second moment, we explore the question of the coherence and compatibility of the pair composed by the abstract ideals of polynomials and multilinear applications.
\end{abstract}

\section{Introduction and background}
There is a large number of classes of operators in the literature, see for example \cite{BBDP07, C14, D03, M03, PR12}. Most of these previous works have followed a very similar script, trying to prove similar properties, of which we can highlight the following: characterize the elements of space by inequalities, build a suitable norm in the space, and then show that the normed space that has just been constructed is an ideal of Banach of multilinear operators. Some works have also explored the concept of the $n$-homogeneous polynomials by seeking the same properties found for the space of multilinear applications.

Faced with so many coincidences, the concern arose to create an abstract class of operators that could generalize as many as possible of those already existing in the literature. Thinking in this direction, D. Serrano-Rodríguez in \cite{S13} introduced the abstract class of multilinear operators $\gamma$-summing. This work shows that this class is a Banach ideal of multilinear applications. However, it should be noted that the work of abstraction is not an easy task. For example, Serrano-Rodríguez's work \cite{S13}  contained small gaps, which were filled by the work of G. Botelho and J. Campos in \cite{BC17}.

Thus, following the natural script, the proposal of this work is, in a first moment, to construct the abstract class of the $n$ -homogeneous polynomials absolutely $\gamma$-summing.

Once the ideal of multilinear applications and the ideal of homogeneous polynomials is constructed, we must consider the following fact that is well-known in the literature: although several common multi-ideals and polynomial ideals are usually associated with the some operator ideal, the extension of an operator ideal to polynomials and multilinear mapping is not always a simple task. For example, the ideal of absolutely summing operators has, at least, eight possible extensions to higher degrees (see, for example, \cite{BPR07, CP07, D03, M03CM, M03, PS11, PG05}, and references therein).

In this way, several authors have started to create tools to evaluate how good a polynomial/multilinear extension of a given operator ideal is. For example, in $\cite{BP05}$, the authors present the concept of ideal closed under differentiation (CUD) and ideal closed for scalar multiplication (CSM). One should also highlight $\cite{CDM09}$, where the concept of coherence and compatibility arose. This concept of coherence and compatibility was improved by Pellegrino and Ribeiro in $\cite{PR14}$, where it was started to work with pairs, and was composed of ideals of polynomials and ideals of multilinear operators. Lately, this has been widely used in the literature and we have now investigated whether the pair of abstract ideals of the $\gamma$-summing multilinear operators and the ideals of $n$-homogeneous $\gamma$-summing polynomials are coherent and compatible, in the sense introduced in the literature by Pellegrino and Ribeiro in $\cite{PR14}$.

We will use the letters $E,E_{1},\dots,E_{n},F,G,H$ to represent Banach spaces over the same scalar-field $\mathbb{K}=\mathbb{R}$ or $\mathbb{C}$. The closed unit ball of $E$ is denoted by $B_E$ and its topological dual by $E'$. We use BAN to denote the class of all Banach spaces over $\mathbb{K}$. Given Banach spaces $E$ and $F$, the symbol $E\overset{1}\hookrightarrow F$ means that $E$ is a linear subspace of $F$ and $\Vert x\Vert_{F}\leq \Vert x\Vert_{E}$ for every $x \in E$. By $c_{00}(E)$ we denote the set of all $E$-valued finite sequences, which, as usual, can be regarded as infinite sequences by completing with zeros. For every $j\in\mathbb{N}$, $e_j = (0,\dots, 0, 1, 0, 0,\dots)$ where $1$ appears at the $j$-th coordinate.

For each positive integer $n$, let $\mathcal{L}_{n}$ denote the class of all continuous $n$-linear operators between Banach spaces. An ideal of multilinear mappings (or multi-ideal) $\mathcal{M}$ is a subclass of the class $\mathcal{L}={\textstyle\bigcup\limits_{n=1}^{\infty}} \mathcal{L}_{n}$ of all continuous multilinear operators between Banach spaces, such that for a positive integer $n$, Banach spaces $E_{1},\ldots,E_{n}$ and $F$, the components
\[
\mathcal{M}_{n}(E_{1},\ldots,E_{n};F):=\mathcal{L}_{n}(E_{1},\ldots
,E_{n};F)\cap\mathcal{M}%
\]
satisfy:

(Ma) $\mathcal{M}_{n}(E_{1},\ldots,E_{n};F)$ is a linear subspace of $\mathcal{L}_{n}(E_{1},\ldots,E_{n};F)$, which contains the $n$-linear mappings of finite type.

(Mb) If $T\in\mathcal{M}_{n}(E_{1},\ldots,E_{n};F)$, $u_{j}\in\mathcal{L}_{1}(G_{j};E_{j})$ for $j=1,\ldots,n$ and $v\in\mathcal{L}_{1}(F;H)$, then
\[
v\circ T\circ(u_{1},\ldots,u_{n})\in\mathcal{M}_{n}(G_{1},\ldots,G_{n};H).
\]

Moreover, $\mathcal{M}$ is a (quasi-) normed multi-ideal if there is a function $\Vert\cdot\Vert_{\mathcal{M}}\colon\mathcal{M}\longrightarrow\lbrack0,\infty)$ satisfying

(M1) $\Vert\cdot\Vert_{\mathcal{M}}$ restricted to $\mathcal{M}_{n}(E_{1},\ldots,E_{n};F)$ is a (quasi-) norm, for all Banach spaces $E_{1},\ldots,E_{n}$ and $F.$

(M2) $\Vert T_{n}\colon\mathbb{K}^{n}\longrightarrow\mathbb{K}:T_{n}(\lambda_{1},\ldots,\lambda_{n})=\lambda_{1}\cdots\lambda_{n}\Vert_{\mathcal{M}}=1$ for all $n$,

(M3) If $T\in\mathcal{M}_{n}(E_{1},\ldots,E_{n};F)$, $u_{j}\in\mathcal{L}_{1}(G_{j};E_{j})$ for $j=1,\ldots,n$ and $v\in\mathcal{L}_{1}(F;H)$, then
\[
\Vert v\circ T\circ(u_{1},\ldots,u_{n})\Vert_{\mathcal{M}}\leq\Vert
v\Vert\Vert T\Vert_{\mathcal{M}}\Vert u_{1}\Vert\cdots\Vert u_{n}\Vert.
\]
When all of the components $\mathcal{M}_{n}(E_{1},\ldots,E_{n};F)$ are complete under this (quasi-) norm, $\mathcal{M}$ is called the (quasi-) Banach multi-ideal. For a fixed multi-ideal $\mathcal{M}$ and a positive integer $n$, the class
\[
\mathcal{M}_{n}:=\cup_{E_{1},...,E_{n},F}\mathcal{M}_{n}\left(  E_{1}%
,...,E_{n};F\right)
\]
is called ideal of $n$-linear mappings.

Analogously, for each positive integer $n$, we can define the polynomial ideal $\mathcal{Q}$. For more details, see \cite{PR14}.  

We will also use the definitions of {\it finitely determined} and {\it linearly stable} sequence classes, which were recently introduced in the literature by Botelho and Campos in \cite{BC17}, as follows.

\begin{definition}
	A class of vector-valued sequences $\gamma_s$, or simply a sequence class $\gamma_s$, is a rule that assigns to each $E \in BAN$ a Banach space $\gamma_s(E)$ of $E$-valued sequences;
	that is, $\gamma_s(E)$ is a vector subspace of $E^{\mathbb{N}}$ with the coordinate wise operations, such that:
	$$c_{00}(E)\subseteq \gamma_s(E)\overset{1}\hookrightarrow\ell_{\infty}(E)\mbox{ and } \Vert e_j\Vert_{\gamma_s(\mathbb{K})}=1\mbox{ for every j}.$$
\end{definition}

A sequence class $\gamma_s$ is {\it finitely determined} if for every sequence $(x_j)_{j=1}^{\infty}\in E^{\mathbb{N}}$, $(x_j)_{j=1}^{\infty}\in \gamma_s(E)$ if, and only if, $\sup_{k}\Vert (x_j)_{j=1}^{k}\Vert_{\gamma_s(E)}<+{\infty}$ and, in this case, $$\Vert (x_j)_{j=1}^{\infty}\Vert_{\gamma_s(E)}=\sup_{k}\Vert (x_j)_{j=1}^{k}\Vert_{\gamma_s(E)}.$$

\begin{definition}{}
	A sequence class $\gamma_s$ is said to be linearly stable if for every $u \in \mathcal{L}(E; F)$ it holds
	
	\begin{equation*}
	\left(u\left(x_j \right)\right)_{j=1}^{\infty} \in \gamma_s(F)
	\end{equation*}
	wherever $\left(x_j \right)_{j=1}^{\infty} \in \gamma_s(E)$ and $\|\hat{u} : \gamma_s(E) \rightarrow \gamma_s(F)\| = \|u\|$.
\end{definition}

Throughout the text, we will also use the following definition that was introduced in \cite{BC17}.

\begin{definition}
	Given sequence classes $\gamma_{s_1},\dots,\gamma_{s_n},\gamma_{s}$, we say that $\gamma_{s_1}(\mathbb{K})\cdots\gamma_{s_n}(\mathbb{K})\overset{1}\hookrightarrow\gamma_{s}(\mathbb{K})$ if $\left(\lambda_{j}^{1}\cdots\lambda_{j}^{n}\right)_{j=1}^{\infty}\in\gamma_{s}(\mathbb{K})$ and
	$$\left\Vert \left(\lambda_{j}^{1}\cdots\lambda_{j}^{n}\right)_{j=1}^{\infty}\right\Vert_{\gamma_{s}(\mathbb{K})}\le\prod_{m=1}^{n}\left\Vert\left(\lambda_{j}^{m}\right)_{j=1}^{\infty}\right\Vert_{\gamma_{s_m}(\mathbb{K})}$$
	whenever $\left(\lambda_{j}^{m}\right)_{j=1}^{\infty}\in\gamma_{s_m}(\mathbb{K})$, $m=1,\dots,n$.
\end{definition}

So, the main goal of this note is to construct an abstract space of $n$-homogeneous polynomials and show that it is an ideal of Banach polynomials.

Since the abstract ideal of multilinear applications \cite{S13} and the abstract ideal of polynomials are now known, we will start to study the coherence and compatibility of the pair $\left(\mathcal{Q},\mathcal{M}\right)$ in the sense of Pellegrino and Ribeiro \cite{PR14}, where $\mathcal{Q}$ is the ideal of polynomials and $\mathcal{M}$ is the ideal of multilinear applications.

\section{Absolutely $\gamma$ - summing polynomials}\label{PGS}

For this study, we will consider sequence class $\gamma_s, \gamma_{s_1},..., \gamma_{s_m}$ to be finitely determined and linearly stable, as defined in \cite{BC17}.


\begin{definition}
Let $E$ and $F$ be Banach spaces. An $m$-homogeneous polynomial $P : E  \longrightarrow F$ is said to be $\gamma_{s, s_1}$ - summing at $a\in E$, if

\begin{equation*}
\left(P(a + x_j) - P(a) \right)_{j=1}^{\infty} \in \gamma_s(F)
\end{equation*}
whenever $\left(x_j  \right)_{j=1}^{\infty} \in \gamma_{s_1}(E)$.
\end{definition}

\begin{remark}
This definition is inspired by the definition of the absolutely $(p,q)$-summing $m$-homogeneous polynomials.

\end{remark}

The space of all $m$-homogeneous polynomials $\gamma_{s, s_1}$ - summing at $a$, as denoted by $\mathcal{P}_{\gamma_{s, s_1}}^{(a)}\left(^mE; F \right)$, is a linear subspace of the $\mathcal{P}\left(^mE; F \right)$. When $a=0$, we write only $\mathcal{P}_{\gamma_{s, s_1}}\left(^mE; F \right)$. The space of all $m$-homogeneous polynomials $\gamma_{s, s_1}$-summing at every point will be denoted by $\mathcal{P}_{\gamma_{s, s_1}}^{(ev)}\left(^mE; F \right)$.

By using the polarization formula \cite[Corollary 1.6]{DJT95} we can easily prove the following result:

\begin{proposition}\label{P1.1.}
$P \in \mathcal{P}_{\gamma_{s, s_1}}^{ev}\left(^mE; F \right)$  if, and only if, $\check{P}$ is $\gamma_{s, s_1}$ - summing in every point $(a_1,\dots, a_m) \in E\times \overset{m}{\cdots} \times E$. 
\end{proposition}

The next result will be used to construct a standard norm at the space of the $m$-homogeneous polynomials $\gamma_{s, s_1}$-summing at the origin, $\mathcal{P}_{\gamma_{s, s_1}}\left(^mE; F \right)$.

\begin{proposition}\label{CaracterizacaoOrigem}
$P \in \mathcal{P}_{\gamma_{s, s_1}}(^mE; F)$ if, and only if, there is a constant $C > 0$, such that

\begin{equation}\label{E1.4}
\left\Vert\left(P( x_j)   \right)_{j=1}^{\infty} \right\Vert_{\gamma_s(F)} \le C\left\Vert \left(x_j  \right)_{j=1}^{\infty}  \right\Vert_{\gamma_{s_1}(E)}^m
\end{equation}
whenever $\left(x_j  \right)_{j=1}^{\infty} \in \gamma_{s_1}(E)$. In addition, the infimum of the constants $C > 0$ satisfying inequality \eqref{E1.4} defines a norm in $ \mathcal{P}_{\gamma_{s, s_1}}(^mE; F)$, as denoted by $\pi(\cdot )$.
\end{proposition}

\begin{proof} 
Suppose $P \in \mathcal{P}_{\gamma_{s, s_1}}(^mE; F)$. From Proposition \ref{P1.1.}, it follows that $\check{P}$ is absolutely $\gamma_{s, s_1}$ - summing at the origin. By \cite[Proposition $2$]{S13}, exists $C > 0$, such that
\begin{equation*}
\left\|\left(\check{P}\left(x_j \right)^m \right)_{j=1}^{\infty} \right\|_{\gamma_s(F)} \le C\left\|\left(x_j \right)_{j=1}^{\infty} \right\|_{\gamma_{s_1}(E)}^m.
\end{equation*}

So,
\begin{equation*}
\left\|\left(P(x_j) \right)_{j=1}^{\infty} \right\|_{\gamma_s(F)} \le C\left\|\left(x_j \right)_{j=1}^{\infty} \right\|_{\gamma_{s_1}(E)}^m.
\end{equation*}

Given that $\gamma_{s}$ and $\gamma_{s_1}$ are finitely determined, the reciprocal is immediate. It is easy to see that $\pi(\cdot )$ define a norm in $ \mathcal{P}_{\gamma_{s, s_1}}(^mE; F)$.
\end{proof}

The following lemma, whose proof can be obtained following Proposition \ref{P1.1.} and \cite[Lemma 2]{BBJP06}, is crucial for the proof of the main result of this section:

\begin{lemma}\label{L1.1.}
If $P \in \mathcal{P}_{\gamma_{s, s_1}}(^mE; F)$ and $a \in E$, then there is a constant $C_a > 0$, such that
\begin{equation*}
\left\Vert\left(P(a + x_j) - P(a) \right)_{j=1}^{\infty}  \right\Vert_{\gamma_s(F)} \le C_a,
\end{equation*}
for all $\left(x_j  \right)_{j=1}^{\infty} \in \gamma_{s_1}(E)$ and $\left\Vert\left(x_j  \right)_{j=1}^{\infty}  \right\Vert_{\gamma_{s_1}(E)} \le 1$.
\end{lemma}

The next result, as in the case of Proposition \ref{CaracterizacaoOrigem}, is a characterization by inequality of the operators in  $\mathcal{P}_{\gamma_{s, s_1}}^{ev}(^mE; F)$. The same is very important because from it we can extract a norm that makes $\mathcal{P}_{\gamma_{s, s_1}}^{ev}(^mE; F)$ a Banach space. The proof was inspired on \cite{BBDP07} and \cite{M03}. 

\begin{theorem}\label{T1.2.}
Let $P \in \mathcal{P}(^mE; F)$. The following assertions are equivalents:

\begin{enumerate}[$(a)$]
\item $P \in \mathcal{P}_{\gamma_{s, s_1}}^{(ev)}(^mE; F)$;

\item There is $C > 0$ satisfying
\begin{equation*}
\left\Vert \left(P(b + x_j) - P(b)  \right)_{j=1}^{n}  \right\Vert_{\gamma_s(F)} \le C \left( \|b\| + \left\Vert\left(x_j  \right)_{j=1}^{n}  \right\Vert_{\gamma_{s_1}(E)} \right)^m
\end{equation*}
for all $n \in \mathbb{N}$ and $x_1,\cdots, x_m, a \in E$.

\item There is $C > 0$ satisfying
\begin{equation}\label{EE1.30}
\left\Vert \left(P(b + x_j) - P(b)  \right)_{j=1}^{\infty}  \right\Vert_{\gamma_s(F)} \le C \left( \|b\| + \left\Vert\left(x_j  \right)_{j=1}^{\infty}  \right\Vert_{\gamma_{s_1}(E)} \right)^m
\end{equation}
for all $b \in E$ and $\left(x_j  \right)_{j=1}^{\infty} \in \gamma_{s_1}(E)$.
\end{enumerate}
\end{theorem}

\begin{proof}
$(c) \Rightarrow (a)$ and $(c) \Rightarrow (b)$ are immediate. Using the fact that the sequence classes considered are finitely determined, it immediately follows that $(b) \Rightarrow (c)$.

Therefore, it remains to prove that $(a) \Rightarrow (c)$.

Let $G = E \times \gamma_{s_1}(E)$. For each $P \in \mathcal{P}_{\gamma_{s, s_1}}^{(ev)}(^mE; F)$, set the following application
\begin{equation*}
\eta_{\gamma_{s, s_1}}(P) : G \longrightarrow \gamma_s(F)
\end{equation*}
given by
\begin{equation*}
\eta_{\gamma_{s, s_1}}(P)\left(\left(b, \left(x_j  \right)_{j=1}^{\infty}  \right)  \right) = \left(P(b +  x_j) - P(b)  \right)_{j=1}^{\infty}.
\end{equation*}

It is not difficult to see that $\eta_{\gamma_{s, s_1}}(P)$ is an $m$-homogeneous polynomial. To show that $\eta_{\gamma_{s, s_1}}(P)$ is continuous, we will consider, for all $k \in \mathbb{N}$ and $\left(x_j  \right)_{j=1}^{\infty} \in \gamma_{s_1}(F)$, the set
\begin{equation*}
F_{k, \left(x_j  \right)_{j=1}^{\infty}} = \left\{b \in E : \left\Vert\eta_{\gamma_{s, s_1}}(P)\left(\left(b, \left(x_j  \right)_{j=1}^{\infty}  \right)  \right)   \right\Vert_{\gamma_s(F)} \le k \right\}.
\end{equation*}
Note that the set $F_{k, \left(x_j  \right)_{j=1}^{\infty}}$ is closed for all $b \in E$ and $\left(x_j  \right)_{j=1}^{\infty} \in B_{\gamma_{s_1}(F)}$. Indeed, for each $n \in \mathbb{N}$, let
\begin{equation*}
F_{k, \left(x_j  \right)_{j=1}^{n}} = \left\{b \in E : \left\Vert\eta_{\gamma_{s, s_1}}(P)\left(\left(b, \left(x_j  \right)_{j=1}^{n}  \right)  \right)   \right\Vert_{\gamma_s(F)} \le k \right\}.
\end{equation*}
So,
\begin{equation}\label{E1.33}
F_{k, \left(x_j  \right)_{j=1}^{\infty}} = \bigcap_{n \in \mathbb{N}}F_{k, \left(x_j  \right)_{j=1}^{n}}.
\end{equation}
For each $\left(x_j  \right)_{j=1}^{\infty} \in B_{\gamma_{s_1}(E)}$, and fixed $k \in \mathbb{N}$, we can define

\begin{equation*}
D_k : E \longrightarrow [0, \infty)
\end{equation*}
given by

\begin{equation*}
D_k(b) = \left\Vert\left(P(b + x_j) - P(b)  \right)_{j=1}^{n}  \right\Vert_{\gamma_s(F)}.
\end{equation*}
It is clear that $D_k$ is a continuous application. So, each $F_{k, \left(x_j  \right)_{j=1}^{n}}$ is closed because

\begin{equation*}
F_{k, \left(x_j  \right)_{j=1}^{n}} = D_k^{-1}([0, k]).
\end{equation*}
Therefore, from \eqref{E1.33} it follows that $F_{k, \left(x_j  \right)_{j=1}^{\infty}}$ is closed because it is the intersection of closed sets.

Let
\begin{equation*}
F_k = \bigcap_{\left(x_j  \right)_{j=1}^{\infty} \in B_{\gamma_{s_1}^u(E)}}F_{k, \left(x_j  \right)_{j=1}^{\infty}}.
\end{equation*}
By the Lemma \eqref{L1.1.} it follows that
\begin{equation*}
E = \bigcup_{k \in \mathbb{N}}F_k.
\end{equation*}
Using the Baire Category Theorem, we know that there is a constant $k_0 \in \mathbb{N}$ such that $F_{k_0}$ has an interior point. The continuity of the application $\eta_{\gamma_{s, s_1}}(P)$ is obtained by repeating the proof of \cite[Proposition 9.3]{BBJP06} (or \cite[Theorem 4.1]{BBDP07}). Therefore,

\begin{align}\label{EE1.38}
\left\Vert \left(P(b + x_j) - P(b)  \right)_{j=1}^{\infty}  \right\Vert_{\gamma_s(F)} &= \left\Vert\eta_{\gamma_{s, s_1}}(P)\left(\left(b, \left(x_j \right)_{j=1}^{\infty}  \right)  \right)  \right\Vert_{\gamma_s(F)}\\
&\le \left\Vert\eta_{\gamma_{s, s_1}}(P)\right\Vert\left( ||b|| + \left\Vert\left(x_j  \right)_{j=1}^{\infty}  \right\Vert_{\gamma_{s_1}(E)} \right)^m  .\nonumber
\end{align}
\end{proof}

By straightforward computations, we can get the following result.

\begin{corollary}\label{C1.1.}
The infimum of the constants $C > 0$ that satisfy the inequality \eqref{EE1.30} defines a norm in $\mathcal{P}_{\gamma_{s, s_1}}^{(ev)}(^mE; F)$, that will be denoted by $\pi^{(ev)}(\cdot)$. 
\end{corollary}

It is not difficult to see that
\begin{remark}\label{O1.4.}
	$\pi^{(ev)}(P) = \left\Vert \eta_{\gamma_{s, s_1}}(P) \right\Vert$.
\end{remark}

An alternative way of constructing a normed space of the polynomials associated by $\prod_{s, s_1}^{(ev)}$ was introduced in \cite{S13} and denoted by $\mathcal{P}_{\prod_{\gamma_{s, s_1}}^{ev}}$, which would be to observe Proposition \ref{P1.1.} and to consider the set
\begin{equation*}
\mathcal{P}_{\prod_{\gamma_{s, s_1}}^{ev}} := \left\{P \in \mathcal{P}\text{; } \check{P}  \text{ is $\gamma_{s, s_1}$ - summing in every point} \right\}.
\end{equation*}
 and, in this set, to use the norm inherited from the ideal of multilinear applications $\prod_{\gamma_{s, s_1}}^{ev}$, that is,  
\begin{equation*}
\left\Vert P\right\Vert_{\mathcal{P}_{\prod_{\gamma_{s, s_1}}^{ev}}}:=\|\check{P}\|_{\prod_{\gamma_{s, s_1}}^{ev}} = \pi_{\gamma_{s, s_1}}^{(ev)}(\check{P}).
\end{equation*}
The advantage of this approach is that it is already established in the literature (see, for example, \cite[page 46]{BBJP06}) that this set, with this norm, is a Banach ideal of $n$-homogeneous polynomials.

But then, one question arises: What is the relationship between the norms $\pi^{(ev)}(P)$ and $\left\Vert P\right\Vert_{\mathcal{P}_{\prod_{\gamma_{s, s_1}}^{ev}}}$? The answer of this question is given in the next proposition.

\begin{proposition}\label{P.1.3.}
	The norm $\pi^{(ev)}(\cdot)$, defined in Corollary \ref{C1.1.}, satisfies the relation
	\begin{equation*}
	\pi^{(ev)}(P) \le \pi_{\gamma_{s, s_1}}^{(ev)}(\check{P}) \le \displaystyle\frac{m^m}{m!}\pi^{(ev)}(P)
	\end{equation*}
	for any $P \in \mathcal{P}_{\gamma_{s, s_1}}^{ev}(^mE; F)$.
\end{proposition}

\begin{proof}
	If $P \in \mathcal{P}_{\gamma_{s, s_1}}^{ev}(^mE; F)$, then, by Proposition \ref{P1.1.}, $\check{P}$ is $\gamma_{s, s_1}$-summing in every point. In this way, for any $\left(x_j \right)_{j=1}^{\infty} \in \gamma_{s_1}(E)$ and $a \in E$, we have
\begin{align*}
\left\|\left(P\left(a + x_j \right) - P(a) \right)_{j=1}^{\infty} \right\|_{\gamma_s(F)} &= \left\|\left(\check{P}\left(a + x_j \right)^m - \check{P}(a)^m \right)_{j=1}^{\infty} \right\|_{\gamma_s(F)}\\
&\le \pi_{\gamma_{s, s_1}}^{(ev)}(\check{P})\left(\|a\| + \left\|\left(x_j \right)_{j=1}^{\infty} \right\|_{\gamma_{s_1}(E)} \right)^m,
\end{align*}
from which it follows that $\pi^{(ev)}(P) \le \pi_{\gamma_{s, s_1}}^{(ev)}(\check{P})$.

For the other inequality, we will use the same tools that appear in the demonstration of  \cite[Theorem $2$]{S13}. Let $G = E \times \gamma_{s_1}(E) $ be gifted with sum norm and $\Phi : \prod_{\gamma_{s, s_1}}^{ev}(E^m; F) \rightarrow \mathcal{L}(G,\overset{m}{\dots}, G; \gamma_s(F))$ be defined by 
$$\Phi(T)\left(\left(a_1,\left(x_j^{(1)}\right)_{j=1}^{\infty}\right),\dots,\left(a_m,\left(x_j^{(m)}\right)_{j=1}^{\infty}\right)\right)=\left(T\left(a_1+x_j^{(1)},\dots,a_m+x_j^{(m)}\right)-T(a_1,\dots,a_m)\right)_{j=1}^{\infty}.$$
In \cite{S13}, we find that $\pi_{\gamma_{s, s_1}}^{ev}(\check{P}) = \|\Phi(\check{P})\|$.

Note that, for any $\left(x_j^{(i)}\right)_{j=1}^{\infty} \in \gamma_{s_1}(E)$ and $\epsilon_i = \pm 1$, $i=1,\dots, m$, we have that $\left(\epsilon_i x_j^{(i)} \right)_{j=1}^{\infty} \in \gamma_{s_1}(E)$. Then, $\left(\epsilon_1 x_j^{(1)} +\cdots + \epsilon_m x_j^{(m)} \right)_{j=1}^{\infty} \in \gamma_{s_1}(E)$ and
\begin{equation*}
\left\|\left(\epsilon_1 x_j^{(1)} +\cdots + \epsilon_m x_j^{(m)} \right)_{j=1}^{\infty} \right\|_{\gamma_{s_1}(E)} \le \left\|\left( x_j^{(1)}  \right)_{j=1}^{\infty} \right\|_{\gamma_{s_1}(E)} +\cdots + \left\|\left( x_j^{(m)} \right)_{j=1}^{\infty} \right\|_{\gamma_{s_1}(E)}.
\end{equation*}
Therefore,
\begin{align*}
&\|\Phi(\check{P})\|\\
&= \sup_{\left\|\left(a_i,\left(x_j^{(i)}\right)_{j=1}^{\infty}\right)\right\|_{G} \le 1} \left\|\Phi(\check{P})\left((a_1, (x_j^{(1)})_{j=1}^{\infty}),\cdots, (a_m, (x_j^{(m)})_{j=1}^{\infty})\right)\right\|_{\gamma_s(F)}\\
&= \sup_{\left\|\left(a_i,\left(x_j^{(i)}\right)_{j=1}^{\infty}\right)\right\|_{G} \le 1} \left\|(\check{P}(a_1 + x_j^{(1)},\dots, a_m + x_j^{(m)})-\check{P}(a_1,\dots, a_m))_{j=1}^{\infty}\right\|_{\gamma_s(F)}\\
&\le \frac{1}{2^mm!} \sup_{\left\|\left(a_i,\left(x_j^{(i)}\right)_{j=1}^{\infty}\right)\right\|_{G} \le 1}\sum_{\epsilon_i = \pm 1}\left\|(P(\epsilon_1(a_1 + x_j^{(1)})+\cdots+ \epsilon_m(a_m + x_j^{(m)}) ) - P\left(\epsilon_1a_1+\cdots + \epsilon_ma_m\right) )_{j=1}^{\infty} \right\|_{\gamma_s(F)}\\
&= \frac{1}{2^mm!} \sup_{\left\|\left(a_i,\left(x_j^{(i)}\right)_{j=1}^{\infty}\right)\right\|_{G} \le 1}\sum_{\epsilon_i = \pm 1}\left\|\eta_{\gamma_{s, s_1}}(P)\left(\left(\epsilon_1a_1+\cdots + \epsilon_ma_m, \left(\epsilon_1x_j^{(1)}+\cdots+ \epsilon_mx_j^{(m)}\right)_{j=1}^{\infty} \right) \right) \right\|_{\gamma_s(F)}\\
&\le \frac{1}{2^mm!} \sup_{\left\|\left(a_i,\left(x_j^{(i)}\right)_{j=1}^{\infty}\right)\right\|_{G} \le 1}\sum_{\epsilon_i = \pm 1} \|\eta_{\gamma_{s, s_1}}(P)\|\left\|\left(\epsilon_1a_1+\cdots + \epsilon_ma_m, \left(\epsilon_1x_j^{(1)}+\cdots+ \epsilon_mx_j^{(m)}\right)_{j=1}^{\infty} \right)\right\|_{G}^m\\
&\le \frac{\|\eta_{\gamma_{s, s_1}}(P)\|}{2^mm!} \sup_{\left\|\left(a_i,\left(x_j^{(i)}\right)_{j=1}^{\infty}\right)\right\|_{G} \le 1}\sum_{\epsilon_i = \pm 1}\left(\|a_1\| + \left\|\left(x_j^{(1)}\right)_{j=1}^{\infty} \right\|_{\gamma_{s_1}(E)} +\cdots + \|a_m\| + \left\|\left(x_j^{(m)}\right)_{j=1}^{\infty} \right\|_{\gamma_{s_1}(E)} \right)^m\\
&= \frac{\|\eta_{\gamma_{s, s_1}}(P)\|}{m!} \sup_{\left\|\left(a_i,\left(x_j^{(i)}\right)_{j=1}^{\infty}\right)\right\|_{G} \le 1}\left(\left\|\left(a_1,\left(x_j^{(1)}\right)_{j=1}^{\infty}\right)\right\|_G +\cdots + \left\|\left(a_m,\left(x_j^{(m)}\right)_{j=1}^{\infty}\right)\right\|_G \right)^m\\
&= \frac{m^m}{m!}\|\eta_{\gamma_{s, s_1}}(P)\|,
\end{align*}
where the function $\eta_{\gamma_{s, s_1}}(P)$ was defined in the proof of Theorem \ref{T1.2.}. Therefore, it follows from Remark \ref{O1.4.} that
\begin{equation*}
\pi^{ev}(P) \le \pi_{\gamma_{s, s_1}}^{ev}(\check{P}) \le \frac{m^m}{m!}\pi^{ev}(P).
\end{equation*}		
\end{proof}

This proposition gives us a relationship that is  satisfactory between $\pi^{ev}(\cdot)$ and $\pi_{\gamma_{s, s_1}}^{ev}(\cdot)$. However, it is important to note that this result was already expected because there is a well-known inequality in the literature that establishes a relationship between the norm of the a $m$-homogeneous polynomial $P$ and the symmetric $m$-linear application associate to $P$, by
\begin{equation*}
\left\Vert P\right\Vert\le\left\Vert\check{P}\right\Vert\le\frac{m^m}{m!}\left\Vert P\right\Vert
\end{equation*}
which was shown in \cite[Theorem 2.2]{mujica}. This same shows that the constant $m^m/m!$ is the best possible solution. For more details to see \cite[example 2I]{mujica}. 

We also emphasize that the result presented in Proposition \ref{P1.1.} is of great importance because through it we have that $\mathcal{P}_{\gamma_{s, s_1}}^{ev}$ is a homogeneous polynomials  ideal. We now need proof that this is a  normed homogeneous polynomials  ideal and complete (Banach), with the norm $\pi^{ev}(\cdot)$.

The next proposition shows that $\pi^{(ev)}(id_{\mathbb{K}}) = 1$. The proof can be obtained by following \cite[Proposition 4.3]{BBJP06} with the necessary adaptations.

\begin{proposition}
Let $id_{\mathbb{K}} : \mathbb{K} \longrightarrow \mathbb{K}$ given by $id_{\mathbb{K}}(x) = x^m$ and suppose that $\gamma_{s_1}(\mathbb{K})\overset{m}{\cdots} \gamma_{s_1}(\mathbb{K}) \overset{1}{\hookrightarrow} \gamma_s(\mathbb{K})$. Then, $id_{\mathbb{K}} \in \mathcal{P}_{\gamma_{s, s_1}}^{ev}(^m\mathbb{K}; \mathbb{K})$ and
\begin{equation*}
\pi^{(ev)}(id_{\mathbb{K}}) = 1.
\end{equation*}
\end{proposition}

The proof of the next Proposition follows the similar result from \cite{BBDP07}.

\begin{proposition}
The linear map

\begin{equation*}
\eta_{\gamma_{s, s_1}} : \mathcal{P}_{\gamma_{s, s_1}}^{(ev)}\left(^mE; F  \right) \longrightarrow \mathcal{P}\left(^mG; \gamma_s(F)  \right)
\end{equation*}
where $G = E \times \gamma_{s_1}(E)$, given by
\begin{equation}
\eta_{\gamma_{s, s_1}}(P)\left(\left(b, \left(x_j  \right)_{j=1}^{\infty}  \right)  \right) = \left(P(b +  x_j) - P(b)  \right)_{j=1}^{\infty}.
\end{equation}
is injective and its range is closed in $\mathcal{P}\left(^mG; \gamma_s(F)  \right)$.
\end{proposition}

So, we easily get the following result:

\begin{proposition}
The space $\mathcal{P}_{\gamma_{s, s_1}}^{(ev)}\left(^mE; F  \right)$ is complete under the norm $\pi^{(ev)}(\cdot )$.
\end{proposition}

\begin{theorem}\label{T1.4.}
$\left(\mathcal{P}_{\gamma_{s, s_1}}^{(ev)}, \pi^{(ev)}(\cdot )  \right)$ is a homogeneous  polynomial ideal Banach between Banach spaces.
\end{theorem}

\begin{proof}
Let $u \in \mathcal{L}(G; E)$, $P \in \mathcal{P}\left(^mE; F  \right)$, $t \in \mathcal{L}(F; H)$ and $a \in G$. Given that $\mathcal{P}_{\gamma_{s, s_1}}^{(ev)}$ is a homogeneous polynomials ideal, then $t \circ P \circ u \in \mathcal{P}_{\gamma_{s, s_1}}^{(ev)}(^mG; H)$. Now, if $\left(x_j  \right)_{j=1}^{\infty} \in \gamma_{s_1}(G)$, then it follows from the linear stability of $\gamma_s$ and $\gamma_{s_1}$ that
\begin{align*}
\left\|\left(t \circ P \circ u(a + x_j) - t \circ P \circ u(a) \right)_{j=1}^{\infty} \right\|_{\gamma_s(H)} 
&\le \|t\|\left\|\left(P \left(u(a + x_j) \right) - P \left(u(a) \right) \right)_{j=1}^{\infty} \right\|_{\gamma_s(F)} \nonumber \\
&= \|t\|\left\|\left(P \left(u(a) + u(x_j) \right) - P \left(u(a) \right) \right)_{j=1}^{\infty} \right\|_{\gamma_s(F)} \nonumber \\
&\le \|t\| \pi^{(ev)}(P) \left(\|u(a)\| + \left\|\left(u(x_j)  \right)_{j=1}^{\infty}  \right\|_{\gamma_{s_1}(E)}  \right)^m \nonumber \\
&\le \|t\| \pi^{(ev)}(P)\|u\|  \left(\|a\| + \left\|\left(x_j  \right)_{j=1}^{\infty}  \right\|_{\gamma_{s_1}(G)}  \right)^m \nonumber. 
\end{align*}
So, $\mathcal{P}_{\gamma_{s, s_1}}^{(ev)}$ satisfies the ideal property and
\begin{equation*}
\pi^{(ev)}(t \circ P \circ u) \le  \|t\| \pi^{(ev)}(P)\|u\| .
\end{equation*}
Therefore, $\left(\mathcal{P}_{\gamma_{s, s_1}}^{(ev)}, \pi^{(ev)}(\cdot) \right)$ is a  homogeneous polynomial ideal Banach between Banach spaces.

\end{proof}

\section{Coherence and compatibility}

In this section, we will study the coherence and the compatibility of the pairs formed by the ideals of $\gamma$-summing multilinear applications and $\gamma$-summing homogeneous polynomials. This concept that was introduced in the literature by Pellegrino and Ribeiro in \cite{PR14}, and their definitions are presented below.


We will consider the sequence $\left(\mathcal{U}_k, \mathcal{M}_k  \right)_{k=1}^N$, where each $\mathcal{U}_k$ is a (quasi-) normed ideal of $k$ - homogeneous polynomials and each $\mathcal{M}_k$ is a (quasi-) normed ideal of $k$ - linear mappings. The parameter $N$ can eventually be infinity.

\begin{definition}[Compatible pair of ideals]
Let $\mathcal{U}$ be a normed operator ideal and $N \in \left(\mathbb{N} - \{1\}  \right)\cup \{\infty \}$. A sequence $\left(\mathcal{U}_n, \mathcal{M}_n  \right)_{n=1}^N$, with $\mathcal{U}_1 = \mathcal{M}_1 = \mathcal{U}$, is compatible with $\mathcal{U}$  if there exists positive constants $\alpha_1, \alpha_2, \alpha_3$ such that for all Banach spaces $E$ and $F$, the following conditions hold for all $n \in \{2,\cdots, N\}:$

\begin{description}
\item $(CP1)$  If $k \in \{1,\dots, n\}$, $T \in \mathcal{M}_n(E_1,\dots, E_n;F)$ and $a_j \in E_j$ for all $j \in \{1,\dots, n\}\setminus\{k\}$, then $ T_{a_1,\dots, a_{k-1},a_{k+1},\dots, a_n} \in \mathcal{U}(E_k; F)$
and
\begin{equation*}
\left\Vert T_{a_1,\dots, a_{k-1},a_{k+1},\dots, a_n} \right\Vert \le \alpha_1 \left\Vert T\right\Vert_{\mathcal{M}_n}\|a_1\|\cdots \|a_{k-1}\| \ \|a_{k+1}\|\cdots \|a_n\|.
\end{equation*}

\item $(CP2)$ If $P \in \mathcal{U}_n(^nE; F)$ and $a \in E$, then $P_{a^{n-1}} \in \mathcal{U}(E; F)$ and
\begin{equation*}
\left\Vert P_{a^{n-1}}\right\Vert_{\mathcal{U}} \le \alpha_2 \max{\left\{\left\Vert\overset{\vee}{P}\right\Vert_{\mathcal{M}_n}, \left\Vert P \right\Vert_{\mathcal{U}_n} \right\}}\Vert a\Vert^{n-1}.
\end{equation*}

\item (CP3) If $u \in \mathcal{U}(E_n; F)$, $\gamma_j \in E'_j$ for all $j = 1,\dots, n-1$, then $\gamma_1 \cdots \gamma_{n-1}u \in \mathcal{M}_n(E_1,\dots, E_n; F)$
and 
\begin{equation*}
\left\Vert\gamma_1 \cdots\gamma_{n-1}u  \right\Vert_{\mathcal{M}_n} \le \alpha_3 \Vert\gamma_1\Vert\cdots\Vert\gamma_{n-1}\Vert\left\Vert u  \right\Vert_{\mathcal{U}}.
\end{equation*}

\item $(CP4)$ If $u \in \mathcal{U}(E; F)$ and $\gamma \in E'$, then $\gamma^{(n-1)}u \in \mathcal{U}_n(^{n}E; F)$.

\item $(CP5)$ $P$ belongs to $\mathcal{U}_n(^nE; F)$ if, and only if, $\overset{\vee}{P}$ belongs to $\mathcal{M}_n(^nE; F)$.
\end{description}

\end{definition}

\begin{definition}[Coherent pair of ideals]\label{D2.2.}
Let $\mathcal{U}$ be a normed operator ideal and let $N \in \mathbb{N} \cup \{\infty \}$. A sequence $\left(\mathcal{U}_k, \mathcal{M}_k  \right)_{k=1}^{N}$, with $\mathcal{U}_1 = \mathcal{M}_1 = \mathcal{U}$, is coherent if there exist positive constants $\beta_1, \beta_2, \beta_3$ such that for all Banach spaces $E$ and $F$ the following conditions hold for $k = 1,\dots,N-1:$

\begin{description}
\item $(CH1)$ If $T \in \mathcal{M}_{k + 1}\left(E_1,\dots, E_{k+1}; F \right)$ and $a_j \in E_j$ for $j=1,\dots, k+1$, then
\begin{equation*}
T_{a_j} \in \mathcal{M}_k\left(E_1,\dots, E_{j-1}, E_{j+1},\dots, E_{k+1}; F  \right)
\end{equation*}
and
\begin{equation*}
\left\|T_{a_j} \right\|_{\mathcal{M}_k} \le \beta_1 \left\| T \right\|_{\mathcal{M}_{k + 1}}\|a_j\|.
\end{equation*}

\item $(CH2)$ If $P \in \mathcal{U}_{k+1}\left(^{k+1}E; F \right)$, $a \in E$, then $P_a$ belongs to $\mathcal{U}_k\left(^kE; F \right)$ and
\begin{equation*}
\left\| P_a \right\|_{\mathcal{U}_k} \le \beta_2 \max{\left\{\left\| \overset{\vee}{P} \right\|_{\mathcal{M}_{k+1}}, \left\| P \right\|_{\mathcal{U}_{k+1}} \right\}}\Vert a \Vert.
\end{equation*}

\item $(CH3)$ If $T \in \mathcal{M}_k(E_1,\dots,E_k; F)$, $\gamma \in E'_{k+1}$, then
\begin{equation*}
\gamma T \in \mathcal{M}_{k + 1}(E_1,\dots,E_{k + 1}; F)
\end{equation*}
and 
\begin{equation*}
\left\|\gamma T\right\|_{\mathcal{M}_{k + 1}} \le \beta_3\Vert\gamma\Vert \left\|T \right\|_{ \mathcal{M}_{k}}.
\end{equation*}

\item $(CH4)$ If $P \in \mathcal{U}_{k}\left(^kE; F \right)$ and $\gamma \in E'$, then $\gamma P \in \mathcal{U}_{k+1}\left(^{k + 1}E; F \right).$

\item $(CH5)$ For all $k=1,\dots,N$, $P$ belongs to $\mathcal{U}_k(^kE; F)$ if, and only if, $\overset{\vee}{P}$ belongs to $\mathcal{M}_k(^kE; F)$.
\end{description}

\end{definition}


In this section, we will denote the Banach $\gamma_{s, s_1}$-summing $m$-linear operators ideal and the Banach $\gamma_{s, s_1}$-summing $m$-homogeneous polynomials ideal by $\left(\prod_{\gamma_{s, s_1}}^{m, (ev)}; \pi_{\gamma_{s, s_1}}^{m, ev}(\cdot)\right)$ and $\left(\mathcal{P}_{\gamma_{s, s_1}}^{m, (ev)}; \pi^{m, ev}(\cdot)\right)$, respectively. The reason for this is to evidence the linearity/homogeneity of the components of the ideal. 

We will study the coherence and the compatibility of the pair $\left(\mathcal{P}_{\gamma_{s, s_1}}^{m, (ev)}, \prod_{\gamma_{s, s_1}}^{m, (ev)}   \right)_{m=1}^{N}$ with the ideal $\prod_{\gamma_{s, s_1}}^{ev}$.


\begin{remark}\label{O2.1.}
For any Banach spaces $E$ and $F$, $ \prod_{\gamma_{s, s_1}}^{1, ev}(E; F) = \mathcal{P}_{\gamma_{s, s_1}}^{1, ev}(E; F)= \prod_{\gamma_{s, s_1}}^{ev}(E; F) $.
\end{remark}

In the next two propositions, we will check the conditions (CH1) and (CH2) of Definition \ref{D2.2.}

\begin{proposition}\label{P2.1.}
For each $T \in \prod_{\gamma_{s, s_1}}^{m+1, ev}(E_1,\dots, E_{m+1}; F)$ 
and $(a_1,\dots, a_{m+1}) \in E_1 \times \cdots \times E_{m+1}$,
\begin{equation*}
T_{a_k}(x_1, \dots , x_{k-1},  x_{k+1}, \dots , x_{m+1}) := T(x_1, \dots , x_{k-1}, a_k , x_{k+1}, \dots , x_{m+1})
\end{equation*}
belongs to $\prod_{\gamma_{s, s_1}}^{m, ev}(E_1,\dots , E_{k-1}, E_{k+1}, \dots, E_{m+1}; F)$ and 
\begin{equation*}
\pi_{\gamma_{s, s_1}}^{m, ev}(T_{a_k}) \le \pi_{\gamma_{s, s_1}}^{m+1, ev}(T) \|a_k\|.
\end{equation*}

\end{proposition}

\begin{proof}
Let $T \in \prod_{\gamma_{s, s_1}}^{m+1, ev}(E_1,\dots, E_{m+1}; F)$ 
and $(a_1,\dots, a_{m+1}) \in E_1 \times \cdots \times E_{m+1}$, $\left(x_j^{(n)} \right)_{j=1}^{\infty} \in \gamma_{s_1}(E_n)$, for $n=1,\dots, k-1, k+1,\dots, m+1$. We will do the computations only for $k = 1$. The remaining cases are similar. Thus, for each $b_i \in E_i$ and $\left(x_j^{i}\right)_{j=1}^{\infty}\in\gamma_{s_1}(E_i)$, $i=2,\dots,m$, consider the null-sequence $\left(x_{j}^{(1)} \right)_{j=1}^{\infty} \in \gamma_{s_1}(E_1)$; that is, $x_{j}^{(1)}=0$ for every $j \in \mathbb{N}$. Then,
\begin{align*}
&\left(T_{a_1}(b_2 + x_j^{(2)}, \dots ,b_{m+1} + x_j^{(m+1)}) - T_{a_1}(b_2,\dots, b_{m+1}) \right)_{j=1}^{\infty}\\
&= \left(T(a_1 + x_j^{(1)}, b_2 + x_j^{(2)}, \dots ,b_{m+1} + x_j^{(m+1)}) - T(a_1, b_2,\dots, b_{m+1}) \right)_{j=1}^{\infty} \in \gamma_s(F).
\end{align*}
Thus, $T_{a_1} \in \prod_{\gamma_{s, s_1}}^{m, ev}(E_2,\dots ,E_{m+1}; F)$. So,
\begin{align*}
&\left\|\left(T_{a_1}(b_2 + x_j^{(2)}, \dots ,b_{m+1} + x_j^{(m+1)}) - T_{a_1}(b_2,..., b_{m+1}) \right)_{j=1}^{\infty} \right\|_{\gamma_s(F)} \le \\
&\le \pi_{\gamma_{s, s_1}}^{m+1, ev}(T)\|a_1\| \left(\|b_2\| + \left\|\left(x_j^{(2)} \right)_{j=1}^{\infty} \right\|_{\gamma_{s_1}(E_1)} \right) \cdots \left(\|b_{m+1}\| + \left\|\left(x_j^{(m+1)} \right)_{j=1}^{\infty} \right\|_{\gamma_{s_1}(E_m)} \right).
\end{align*}
Therefore,
\begin{equation*}
\pi_{\gamma_{s, s_1}}^{m, ev}(T_{a_1}) \le \pi_{\gamma_{s, s_1}}^{m+1, ev}(T) \|a_1\|.
\end{equation*}

\end{proof}

\begin{proposition}\label{P2.2.}
For each $P \in \mathcal{P}_{\gamma_{s, s_1}}^{m+1, ev}(^{m+1}E, F)$ and $a \in E$,
\begin{equation*}
P_a(x) := \check{P}(a, x, \overset{m}{\dots} , x) 
\end{equation*}
belongs to $\mathcal{P}_{\gamma_{s, s_1}}^{m, ev}(^{m}E, F)$ and
\begin{equation*}
\pi^{m, ev}(P_a) \le \pi_{\gamma_{s, s_1}}^{m+1, ev}(\check{P}) \|a\|.
\end{equation*}

\end{proposition}

\begin{proof}
Let $P \in \mathcal{P}_{\gamma_{s, s_1}}^{m+1, ev}(^{m+1}E, F)$ and $a \in E$. For any $b \in E$ and $\left(x_j \right)_{j=1}^{\infty} \in \gamma_{s_1}(E)$,
it follows from Proposition \ref{P1.1.} that $\check{P}$ is absolutely $\gamma_{s, s_1}$-summing in every point $(a_1,\dots, a_m) \in E\times \overset{m}{\cdots} \times E$. Again, consider the sequence $\left(y_{j} \right)_{j=1}^{\infty} \in \gamma_{s_1}(E_1)$, such that  $y_{j}=0$ for every $j \in \mathbb{N}$. Thus,
\begin{align*}
\left(P_a(b + x_j) - P_a(b)  \right)_{j=1}^{\infty} &= \left(\check{P}(a + y_j, b + x_j, \overset{m}{\dots}, b + x_j) - \check{P}(a, b,\overset{m}{\dots}, b )  \right)_{j=1}^{\infty} \in \gamma_s(F).
\end{align*}
Thus, $P_a \in \mathcal{P}_{\gamma_{s, s_1}}^{m, ev}(^{m}E, F)$. Furthermore,
\begin{align*}
\left\|\left(P_a(b + x_j) - P_a(b)  \right)_{j=1}^{\infty} \right\|_{\gamma_s(F)} \le \pi_{\gamma_{s, s_1}}^{m+1, ev}(\check{P}) \|a\| \left(\|b\| + \left\|\left(x_j \right)_{j=1}^{\infty} \right\|_{\gamma_{s_1}(E)} \right)^m.
\end{align*}
Therefore, 
\begin{equation*}
\pi^{m, ev}(P_a) \le \pi_{\gamma_{s, s_1}}^{m+1, ev}(\check{P}) \|a\|.
\end{equation*}
\end{proof}

The next definition contains an important property that will be used to prove (CH3) and (CH4) of Definition \ref{D2.2.}.

\begin{definition}
	Let $E$ be a Banach space and $\gamma_{s}$ be a sequence class. We say that the sequence class $\gamma_{s}$ is $\mathbb{K}$-closed when, for any $\left(x_j\right)_{j=1}^{\infty}\in\gamma_{s}\left(\mathbb{K}\right)$ and $\left(y_j\right)_{j=1}^{\infty}\in\gamma_{s}\left(E\right)$, the sequence $\left(z_j\right)_{j=1}^{\infty}\in\gamma_{s}\left(E\right)$, where $z_{j}=x_{j}y_{j}$ and
	$$
	\left\Vert\left(z_{j}\right)_{j=1}^{\infty}\right\Vert_{\gamma_{s}\left(E\right)}\le
	\left\Vert\left(x_{j}\right)_{j=1}^{\infty}\right\Vert_{\gamma_{s}\left(\mathbb{K}\right)}
	\left\Vert\left(y_{j}\right)_{j=1}^{\infty}\right\Vert_{\gamma_{s}\left(E\right)}
	$$	
\end{definition}

\begin{example}
	The sequence classes $\ell_p\langle \cdot\rangle$, $\ell_p(\cdot)$, $\ell_p^{mid}(\cdot)$ and $\ell_p^w(\cdot)$ are $\mathbb{K}$-closed.
\end{example}

\begin{definition}\label{DFC}
Let $\gamma_s$ and $\gamma_{s_1}$ be sequence classes. We say that $\gamma_s$ and $\gamma_{s_1}$ are finitely coincident, when $\gamma_{s}(E)=\gamma_{s_1}(E)$ and that
\begin{equation*}
\left\|\left(x_j \right)_{j=1}^{\infty} \right\|_{\gamma_{s}(E)} = \left\|\left(x_j \right)_{j=1}^{\infty} \right\|_{\gamma_{s_1}(E)}
\end{equation*}
for any finite-dimensional linear space $E$.
\end{definition}

\begin{remark}
In the next two propositions we will assume that the sequence class $\gamma_{s}$ is $\mathbb{K}$-closed and $\gamma_s$ and $\gamma_{s_1}$ are finitely coincident.
\end{remark}



\begin{proposition}\label{P2.3.}
Let $T \in \prod_{\gamma_{s, s_1}}^{m, ev}(E_1,..., E_m; F)$ and $\varphi \in E_{m+1}'$, so $\varphi T \in \prod_{\gamma_{s, s_1}}^{m+1,ev}(E_1,..., E_{m+1}; F)$
and

\begin{equation*}
\pi_{\gamma_{s, s_1}}^{m+1, ev}(\varphi T) \le \|\varphi \| \pi_{\gamma_{s, s_1}}^{m, ev}(T).
\end{equation*}
\end{proposition}

\begin{proof}
We will do only the case $m = 2$. The other cases are analogous. Let
$T \in \prod_{\gamma_{s, s_1}}^{2, ev}(E_1, E_2; F)$, $\varphi \in E_{3}'$ and $\left(x_j^{(i)} \right)_{j=1}^{\infty} \in \gamma_{s_1}(E_i)$, $a_i \in E_i$, $i=1, 2, 3$. Thus, because $\gamma_s$ is linearly stable, finitely determined, $\mathbb{K}$-closed, and because $\gamma_s$ and $\gamma_{s_1}$ are finitely coincident, it follows immediately that $$\left(\varphi T\left(a_1 + x_j^{(1)}, a_2 + x_j^{(2)}, a_3 + x_j^{(3)}\right) - \varphi T(a_1, a_2, a_3) \right)_{j=1}^{\infty}\in \gamma_s(F)$$
and that,
\begin{align*}
&\left\|\left(\varphi T\left(a_1 + x_j^{(1)}, a_2 + x_j^{(2)}, a_3 + x_j^{(3)}\right) - \varphi T(a_1, a_2, a_3) \right)_{j=1}^{\infty} \right\|_{\gamma_s(F)}\\
\le& \left\|\left(\varphi(a_3)T\left(a_1, x_j^{(2)}\right)\right)_{j=1}^{\infty}\right\|_{\gamma_s(F)} + \left\|\left(\varphi(a_3)T\left(x_j^{(1)}, a_2 \right)\right)_{j=1}^{\infty}\right\|_{\gamma_s(F)} + \left\| \left(\varphi(a_3)T\left(x_j^{(1)}, x_j^{(2)} \right)\right)_{j=1}^{\infty}\right\|_{\gamma_s(F)} +\\
+& \left\|\left(\varphi\left(x_j^{(3)}\right)T\left(a_1, x_j^{2} \right) \right)_{j=1}^{\infty} \right\|_{\gamma_s(F)} +  \left\|\left(\varphi\left(x_j^{(3)}\right)T\left(x_j^{(1)}, a_2 \right) \right)_{j=1}^{\infty} \right\|_{\gamma_s(F)} +\\
+& \left\|\left(\varphi\left(x_j^{(3)}\right)T\left(x_j^{(1)}, x_j^{2} \right) \right)_{j=1}^{\infty} \right\|_{\gamma_s(F)} + \left\|\left(\varphi\left(x_j^{(3)}\right)T\left(a_1, a_2 \right) \right)_{j=1}^{\infty} \right\|_{\gamma_s(F)}\\
\le& \left\|\left(T\left(a_1, x_j^{(2)} \right) \right)_{j=1}^{\infty} \right\|_{\gamma_s(F)} \left(|\varphi(a_3)| + \left\|\left(\varphi\left(x_j^{(3)} \right) \right)_{j=1}^{\infty} \right\|_{\gamma_{s_1}(\mathbb{K})} \right) +\\
+& \left\|\left(T\left(x_j^{(1)}, a_2 \right) \right)_{j=1}^{\infty} \right\|_{\gamma_s(F)}\left(|\varphi(a_3)| + \left\|\left(\varphi\left(x_j^{(3)} \right) \right)_{j=1}^{\infty} \right\|_{\gamma_{s_1}(\mathbb{K})} \right) +\\
+& \left\|\left(T\left(x_j^{(1)}, x_j^{(2)} \right) \right)_{j=1}^{\infty} \right\|_{\gamma_s(F)} \left(|\varphi(a_3)| + \left\|\left(\varphi\left(x_j^{(3)} \right) \right)_{j=1}^{\infty} \right\|_{\gamma_{s_1}(\mathbb{K})} \right) + \left\|T\left(a_1, a_2 \right) \right\|\left\|\left(\varphi\left(x_j^{(3)} \right) \right)_{j=1}^{\infty} \right\|_{\gamma_{s_1}(\mathbb{K})}\\
\le& \pi_{\gamma_{s, s_1}}^{ev}(T)\|a_1\|\left\|\left(x_j^{(2)} \right)_{j=1}^{\infty} \right\|_{\gamma_{s_1}(E_2)}\|\varphi\| \left(\|a_3\| + \left\|\left(x_j^{(3)}  \right)_{j=1}^{\infty} \right\|_{\gamma_{s_1}(E_3)} \right) +\\
+& \pi_{\gamma_{s, s_1}}^{ev}(T)\|a_2\|\left\|\left(x_j^{(1)} \right)_{j=1}^{\infty} \right\|_{\gamma_{s_1}(E_1)}\|\varphi\|\left(\|a_3\| + \left\|\left(x_j^{(3)}  \right)_{j=1}^{\infty} \right\|_{\gamma_{s_1}(E_3)} \right) +\\
+& \pi_{\gamma_{s, s_1}}^{ev}(T)\left\|\left(x_j^{(1)} \right)_{j=1}^{\infty} \right\|_{\gamma_{s_1}(E_1)}\left\|\left(x_j^{(2)} \right)_{j=1}^{\infty} \right\|_{\gamma_{s_1}(E_2)}\|\varphi\|\left(\|a_3\| + \left\|\left(x_j^{(3)}  \right)_{j=1}^{\infty} \right\|_{\gamma_{s_1}(E_3)} \right) +\\
+& \pi_{\gamma_{s, s_1}}^{ev}(T)\|a_1\|\|a_2\|\|\varphi\|\left\|\left(x_j^{(3)} \right)_{j=1}^{\infty} \right\|_{\gamma_{s_1}(E_3)}\\
&= \|\varphi \|\pi_{\gamma_{s, s_1}}^{ev}(T)\left( \prod_{i=1}^3\left(\|a_i\| + \left\|\left(x_j^{(i)} \right)_{j=1}^{\infty} \right\|_{\gamma_{s_1}(E_i)} \right) - \|a_1\| \|a_2\| \|a_3\|\right).\\
\end{align*}
Consequently,

\begin{align*}
&\left\|\left(\varphi T\left(a_1 + x_j^{(1)}, a_2 + x_j^{(2)}, a_3 + x_j^{(3)}\right) - \varphi T(a_1, a_2, a_3) \right)_{j=1}^{\infty} \right\|_{\gamma_s(F)}\\
&\le \|\varphi \|\pi_{\gamma_{s, s_1}}^{ev}(T) \prod_{i=1}^3\left(\|a_i\| + \left\|\left(x_j^{(i)} \right)_{j=1}^{\infty} \right\|_{\gamma_{s_1}(E_i)} \right).
\end{align*}
From where it follows that $\varphi T \in \prod_{\gamma_{s, s_1}}^{m+1, ev}(E_1,\dots, E_{m+1}; F)$ and

\begin{equation*}
\pi_{\gamma_{s, s_1}}^{m+1, ev}(\varphi T) \le \|\varphi \| \pi_{\gamma_{s, s_1}}^{m, ev}(T).
\end{equation*}

\end{proof}

\begin{proposition}\label{P2.4.}
Let $P \in \mathcal{P}_{\gamma_{s, s_1}}^{m, ev}(^{m}E, F)$ and $\varphi \in E'$. Then $ \varphi P \in \mathcal{P}_{\gamma_{s, s_1}}^{m+1, ev}(^{m+1}E; F) $
and
\begin{equation*}
\pi^{m+1, ev}(\varphi P) \le \|\varphi \| \pi_{\gamma_{s, s_1}}^{m, ev}(P).
\end{equation*}
\end{proposition}

\begin{proof}
Let $P \in \mathcal{P}_{\gamma_{s, s_1}}^{m, ev}(^{m}E, F)$, $\varphi \in E'$ and $(x_j)_{j=1}^{\infty} \in \gamma_{s_1}(E)$. It is easy to see what is required. To illustrate this point, we will make the case $m = 2 $. The general case is analogous.  For any $a \in E$, because $\gamma_s$ is linearly stable, finitely determined, $\mathbb{K}$-closed, and because $\gamma_s$ and $\gamma_{s_1}$ are finitely coincident. From Proposition \ref{P1.1.}, $\check{P} \in \prod_{\gamma_{s, s_1}}^{ev}(E^2; F)$. So,
\begin{align*}
 &\left( \varphi P(a + x_j) - \varphi P(a)\right)_{j=1}^{\infty}\\
 &= \left(\varphi(a + x_j)P\left(a + x_j \right) - \varphi(a)P(a) \right)_{j=1}^{\infty}\\
 &= \left(\varphi(a + x_j)\check{P}\left(a + x_j, a + x_j \right) - \varphi(a)\check{P}(a, a) \right)_{j=1}^{\infty}\\
 &= 2\left(\varphi(a)\check{P}(a, x_j) \right)_{j=1}^{\infty}  + \left(\varphi(a)\check{P}(x_j, x_j) \right)_{j=1}^{\infty} + \left(\varphi(x_j)\check{P}(a, a) \right)_{j=1}^{\infty} + 2\left(\varphi(x_j)\check{P}(a, x_j) \right)_{j=1}^{\infty} + \\
 &+ \left(\varphi(x_j)\check{P}(x_j, x_j) \right)_{j=1}^{\infty} \in \gamma_s(F).
\end{align*}
Then, $ \varphi P \in \mathcal{P}_{\gamma_{s, s_1}}^{3, ev}(^3E; F) $. Furthermore,
\begin{align*}
&\left\|\left(\varphi P\left(a + x_j \right) - \varphi P(a) \right)_{j=1}^{\infty} \right\|_{\gamma_s(F)} \\
&\le 2\left\|\left(\varphi(a)\check{P}(a, x_j) \right)_{j=1}^{\infty} \right\|_{\gamma_s(F)} + \left\| \left(\varphi(a)\check{P}(x_j, x_j) \right)_{j=1}^{\infty} \right\|_{\gamma_s(F)} + \left\|\left(\varphi(x_j)\check{P}(a, a) \right)_{j=1}^{\infty} \right\|_{\gamma_s(F)} +\\
&+ 2\left\|\left(\varphi(x_j)\check{P}(a, x_j) \right)_{j=1}^{\infty} \right\|_{\gamma_s(F)} + \left\|\left(\varphi(x_j)\check{P}(x_j, x_j) \right)_{j=1}^{\infty} \right\|_{\gamma_s(F)}\\
&\le 2\left\|\left(\check{P}(a, x_j) \right)_{j=1}^{\infty} \right\|_{\gamma_s(F)}\left(|\varphi(a)| + \left\|\left(\varphi(x_j) \right)_{j=1}^{\infty} \right\|_{\gamma_{s_1}(\mathbb{K})} \right)\\
&+ \left\|\left(\check{P}(x_j, x_j) \right)_{j=1}^{\infty} \right\|_{\gamma_s(F)}\left(|\varphi(a)| + \left\|\left(\varphi(x_j) \right)_{j=1}^{\infty} \right\|_{\gamma_{s_1}(\mathbb{K})} \right) + \|\check{P}(a, a)\|\left\|\left(\varphi(x_j) \right)_{j=1}^{\infty} \right\|_{\gamma_{s_1}(\mathbb{K})}\\
\end{align*}
How $\|\check{P}\| \le \pi_{\gamma_{s, s_1}}^{2, ev}(\check{P})$. We have that
\begin{align*}
&\left\|\left(\varphi P\left(a + x_j \right) - \varphi P(a) \right)_{j=1}^{\infty} \right\|_{\gamma_s(F)} \\
&\le \pi_{\gamma_{s, s_1}}^{2, ev}(\check{P})\|\varphi\| \left(3\|a\|^2\left\|\left(\varphi(x_j) \right)_{j=1}^{\infty} \right\|_{\gamma_{s_1}(\mathbb{K})} + 3 \|a\|\left\|\left(\varphi(x_j) \right)_{j=1}^{\infty} \right\|_{\gamma_{s_1}(\mathbb{K})}^2 + \left\|\left(\varphi(x_j) \right)_{j=1}^{\infty} \right\|_{\gamma_{s_1}(\mathbb{K})}^3 \right)\\
&\le \pi_{\gamma_{s, s_1}}^{2, ev}(\check{P})\|\varphi\|\left(\|a\| + \left\|\left(\varphi(x_j) \right)_{j=1}^{\infty} \right\|_{\gamma_{s_1}(\mathbb{K})}  \right)^3.
\end{align*}
Therefore,
\begin{equation*}
\pi^{3, ev}(\varphi P) \le \|\varphi \| \pi_{\gamma_{s, s_1}}^{2, ev}(\check{P}).
\end{equation*}

\end{proof}

By Propositions \ref{P2.1.}, \ref{P2.2.}, \ref{P2.3.}, \ref{P2.4.} and \ref{P1.1.}, the pair $$\left(\left(\mathcal{P}_{\gamma_{s, s_1}}^{m, ev}, \pi^{m, ev}(\cdot ) \right), \left( \mathcal{\prod}_{\gamma_{s, s_1}}^{m, ev}, \pi_{\gamma_{s, s_1}}^{m, ev}(\cdot ) \right)\right)_{m=1}^{\infty}$$ is coherent. Since $\beta_1 = 1, \beta_2 = 1$ and $\beta_3 = 1$, it follows by \cite[Remark 3.3]{PR14} that the pair $\left(\left(\mathcal{P}_{\gamma_{s, s_1}}^{m, ev}, \pi^{m+1, ev}(\cdot ) \right), \left( \prod_{\gamma_{s, s_1}}^{m, ev}, \pi_{\gamma_{s, s_1}}^{m, ev}(\cdot ) \right)\right)_{m=1}^{\infty}$ is compatible with $\prod_{\gamma_{s, s_1}}$. So, we have the following result.

\begin{theorem}
 The sequence $\left(\left(\mathcal{P}_{\gamma_{s, s_1}}^{m, ev}, \pi^{m+1, ev}(\cdot ) \right), \left( \prod_{\gamma_{s, s_1}}^{m, ev}, \pi_{\gamma_{s, s_1}}^{m, ev}(\cdot ) \right)\right)_{m=1}^{\infty}$ is coherent and compatible with $\prod_{\gamma_{s, s_1}}$.
\end{theorem}

It is important to point out that to obtain the proof of Proposition \ref{P2.1.}, \ref{P2.2.} and \ref{P1.1.} it is only necessary that the sequence class be linear stability and finitely determined. However, to demonstrate propositions\ref{P2.3.} and \ref{P2.4.}, extra properties were required for the classes involved; more specifically, the sequence classes should be finitely coincidents and the arrival sequence class should be $\mathbb{K}$-closed. These conditions do not appear to be very restrictive because the main classes of the summing ideals existing in the literature are recovered by our work. The next section illustrates our arguments.

\section{Applications}

For any Banach space $E$, we will denote $\ell_p\langle E \rangle, \ell_p(E)$ and $\ell_p^{w}(E)$ the spaces of Cohen strongly $p$-summing, absolutely $p$-summing and weakly p-summable $E$-valued sequences, respectively. In 2014  S. Kinha and D. Sinha \cite{KS14} introduced the space $\ell_p^{mid}(E)$, which was studied in more details by G. Botelho, J. Campos and J. Santos in \cite{BCS17}. These papers established the inclusions
\begin{equation}\label{Inclusoes}
\ell_p\langle E \rangle \subset \ell_p(E) \subset \ell_p^{mid}(E) \subset \ell_p^w(E).
\end{equation}

The nature of many operators in the literature is to "improve" the convergence of series. For example, we can cite the absolutely summing operators that transform weakly p-summable sequences into absolutely p-summable sequences. Thinking in this direction, we can define several classes of operators that improve the convergence of the series. In the next two examples, we will present classes that are already known and which are particular cases of our work. In the other examples, we present a few classes of operators that are not yet available in the literature, although they can easily be obtained through the construction presented in this work.

\begin{example}
	Let $\mathcal{P}_{(p, q)}^{m, ev}$ be the space of absolutely summing $n$-homogeneous polynomials and $\prod_{(p, q)}^{m, ev}$ be the space of absolutely summing multilinear operators. For more details about this class, see \cite{BBDP07}. Given that the sequences classes involved are $\ell_p^{w}(\cdot)$, and $\ell_p(\cdot)$ and they are linearly stable, finitely determined and, moreover, finitely coincident, and $\ell_p(\cdot)$ is $\mathbb{K}$-closed, it immediately follows that the pair $\left(\left(\mathcal{P}_{(p, q)}^{m, ev}, \|\cdot \|_{ev^2} \right), \left(\prod_{(p, q)}^{m, ev}, \|\cdot \|_{ev^2(p, q)} \right)  \right)_{m=1}^{\infty}$ is coherent and compatible with $\prod_{(p, q)}$.
\end{example}


\begin{example}
	Let $\mathcal{P}_{Coh, p}^{m, ev}$ be the space of the $n$-homogeneous polynomials Cohen strongly $p$-summing everywhere and $\mathcal{L}_{Coh, p}^{m, ev}$ be the space of the multilinears operators Cohen strongly $p$-summing everywhere. For more details about this class, see \cite{C13-tese,S13}. Given that the sequence classes involved are $\ell_p\langle\cdot\rangle$ and $\ell_p(\cdot)$, and because they are linearly stable, finitely determined, finitely coincident and, moreover, $\ell_p\langle\cdot\rangle$ is $\mathbb{K}$-closed, it immediately follows that the pair $\left(\left(\mathcal{P}_{Coh, p}^{m, ev}, \pi^{m, ev} \right), \left(\mathcal{L}_{Coh, p}^{m, ev}, \pi_{Coh, p}^{m, ev} \right) \right)_{m=1}^{\infty}$ is coherent and compatible with $\mathcal{D}_{p}$.
\end{example}

Note that, the classes of multilinears operator/ homogeneous polynomials that are defined in these two examples consider applications that transform weakly strongly $p$-summable sequences in strongly $p$-summable sequences and strongly $p$-summable sequences in Cohen strongly $p$-summable. However, due to the inclusions given above \ref{Inclusoes}, we can present several classes of multilinear operators and homogeneous polynomials that are yet not found in the literature. In this way, this approach establishes an interesting result for this classes. We can consider these examples:

\begin{example}
The classes of multilinears operators and homogeneous polynomials that transform mid $p$-summable sequences in strongly $p$-summable operators. In other words, to consider $\gamma_s = \ell_p$ and $\gamma_{s_1} = \ell_p^{mid}$.  We denote these classes as multilinear operators and homogeneous polynomials mid strongly $p$-summing.
\end{example}

\begin{example}
The classes of multilinears operators and homogeneous polynomials that transform weakly absolutely $p$-summable sequences in mid $p$-summable operators. In other words, to consider $\gamma_s = \ell_p^{mid}$ and $\gamma_{s_1} = \ell_p^{w}$. We denote these classes as multilinear operators and homogeneous polynomials  weakly mid $p$-summing.
\end{example}

\begin{example}
The classes of multilinears operators and homogeneous polynomials that transform weakly strongly $p$-summable and mid $p$-sommable  sequences in Cohen strongly $p$-summable operators. In other words, to consider $\gamma_s = \ell_p\langle \cdot \rangle$ and $\gamma_{s_1} = \ell_p^w, \ell_p^{mid}$. We denote these classes as multilinear operators and homogeneous polynomials  weakly Cohen $p$-summing and mid  Cohen  $p$-summing, respectively.
\end{example}

\section*{Acknowledgments}
The authors thanks Geraldo Botelho for their several helpful conversations and suggestions.

\end{document}